\documentclass[a4paper,12pt,draft]{article}

\evensidemargin
-1cm

\usepackage[top=1in,left=1in,right=1in,bottom=1.5in]{geometry}

\usepackage[T1]{fontenc}
\usepackage{amsmath,amssymb,amsfonts,amsthm,mathrsfs,eucal,dsfont,verbatim}

\theoremstyle{plain}
\newtheorem{theorem}{Theorem}
\newtheorem{lemma}[theorem]{Lemma}
\newtheorem{proposition}[theorem]{Proposition}

\theoremstyle{definition}
\newtheorem{remark}{Remark}

\newtheorem{example}{Example}

\usepackage{hyperref}
\usepackage{eucal,verbatim,mathrsfs,dsfont}

\usepackage{amssymb}
\usepackage{amsmath}
\usepackage{amscd}

\makeindex             

\def\eps{\varepsilon}

\renewcommand{\Re}{\mathds R}

\newcommand{\Mf}{\mathcal{M}}

\newcommand{\Ld}{\mathcal{L}\!\hbox{{\it d}}}
\newcommand{\real}{\mathds{R}}

\begin{document}

\title{Asymptotic behaviour of the
distribution density of the fractional L\'evy motion}

\author{%
    \textsc{Victoria Knopova}%
    \thanks{  V.M.\ Glushkov Institute of Cybernetics,
            NAS of Ukraine,
            40, Acad.\ Glushkov Ave.,
            03187, Kiev, Ukraine,
            \texttt{vic\underline{ }knopova@gmx.de}}
    \textrm{\ \ and\ \ }
    \stepcounter{footnote}\stepcounter{footnote}\stepcounter{footnote}
    \stepcounter{footnote}\stepcounter{footnote}%
    \textsc{Alexei Kulik}%
    \thanks{Institute of Mathematics, NAS of Ukraine, 3, Tereshchenkivska str., 01601  Kiev, Ukraine,
    \texttt{kulik@imath.kiev.ua}}
    }

\date{}

\maketitle

\begin{abstract}
We investigate the  distribution properties of the fractional L\'evy motion  defined by the Mandelbrot-Van Ness representation
$$
Z^H_t :=\int_\Re f(t,s) dZ_s,
$$
where $Z_s$, $s\in \Re$, is a (two-sided) real-valued L\'evy process,  and
   $$
    f(t,s) := \frac{1}{\Gamma(H+1/2)} \left[(t-s)_+^{H-1/2}-(-s)^{H-1/2}_+\right], \quad t,s\in \Re.
 $$
    We consider separately the  cases $0<H<1/2$ (\emph{short memory}) and $1/2<H<1$ (\emph{long memory}), where $H$ is the Hurst parameter, and present the asymptotic behaviour  of  the distribution density of the process. Some  examples are provided, in which it is shown that the behaviour of the density in the cases $0<H<1/2$  and $1/2<H<1$ is completely different.
\end{abstract}


\section{Introduction}

In this paper we consider the distribution properties of the \emph{fractional L\'evy motion} (FLM, in the sequel). Various versions of the FLM have been used in a number of recent publications in order to interpret some experimental data. Apart from  the rigorous mathematical definition, some  modifications   of  the FLM  are derived  from the physical point of view,  see, example, \cite{H99},  \cite{CSC09}, \cite{ES12}.  The FLM driven by an $\alpha$-stable L\'evy process is used as a model for describing sub-diffusive effects in physic and biology  (see \cite{BW10}, \cite{WBMW05}), signal and traffic modeling \cite{LLHD02}, \cite{KM96}, \cite{DM05},  finance \cite{BSV08}, geophysics \cite{PP94}, \cite{W05}, \cite{GLS06}. We refer to \cite{W05} for the discussion in which type of problems the FLM gives an adequate description for the observed phenomena. In the papers quoted above it was shown that the respective phenomena existing in nature can be better described by models, containing the FLM rather than  the fractional Brownian motion (FBM). Finally, we refer to \cite{CG00} for simulations of the FLM, which can be convenient in practical problems.

Similarly to the FBM, the FLM can be defined in two different ways:  via the Mandel-brot-Van Ness representation (see \cite{BCI04} and \cite{M06}), or via the Molchanov-Golosov representation (see \cite{BM09}).  We also refer to  \cite{ST94} for a bit different definition of the fractional stable motion.  These two representations, being equivalent in the Gaussian setting, in the L\'evy setting lead, in general, to different processes;  see \cite{TM11}. Note that, in contrast to the FBM, in some cases the FLM can even be a semi-martingale (\cite{BP09}, \cite{BLS12}).

  In this paper we focus on the FLM  $Z^H_t$ defined by the Mandelbrot-Van Ness representation, i.e.
    \begin{equation}
    Z^H_t :=\int_\Re f(t,s) dZ_s, \label{yt}
    \end{equation}
where $Z_s$, $s\in \Re$, is a (two-sided) real-valued L\'evy process, $H\in (0,1)$ is the \emph{Hurst parameter}, and
    \begin{equation}
    f(t,s) := \frac{1}{\Gamma(H+1/2)} \left[(t-s)_+^{H-1/2}-(-s)^{H-1/2}_+\right], \quad t,s\in \Re,\label{f1}
    \end{equation}
 where  $x_+:=\max(x,0)$. This definition gives a particularly important representative of the class of so-called \emph{moving-average fractional L\'evy motions}.  Since  the FLM, according to the survey above, is an adequate model to some phenomena in nature, it would be appropriate to investigate deeply its properties. In particular, knowledge of the distribution properties of the FLM would naturally make it possible to solve various problems related to statistical inference, simulation, etc.

In this paper we concentrate on the asymptotic behaviour of the distribution density of the FLM. In contrast to the FBM case, the study  of the distribution density of the FLM is much more complicated. In the recent paper \cite{KK11} we presented the investigation of the distribution density of a FLM in the following cases: (i) $H=1/2$, which means merely $Z^H_t\equiv Z_t$, and  (ii) $1/2<H<1$, which corresponds to the so-called \emph{long memory case}, see Definition 1.1 in \cite{M06}. Both cases can be treated in a unified way using a general result about the asymptotic behaviour of  distribution densities of L\'evy driven stochastic integrals with deterministic kernels, see Theorem 2.1 in \cite{KK11}. Since  this theorem requires the respective kernel to be bounded, and the kernel (\ref{f1}) is unbounded when $0<H<1/2$, the case (iii) $0<H<1/2$ cannot be treated in the same way as in \cite{KK11}, and thus requires a completely different approach, which we present below. Further, we show that there is a substantial difference in the behaviour of the density in  the cases (i), (ii) on one hand,  and   the case (iii) on the other hand, namely, the distribution density in these  situations   exhibits absolutely different types of the asymptotic behaviour. We also emphasize, that in contrast to the situation studied in \cite{KK11}, for the case $0<H<1/2$  we do not require the existence of exponential moments of the tails of the L\'evy measure.

Let us outline the rest of the paper. In Section~\ref{s2} we set the notation and formulate our main results, Theorems~\ref{t-main1} and \ref{t-main2}. In Section \ref{s3} we formulate the general result,  Theorem~\ref{td1}, on the asymptotic behaviour of the distribution density of L\'evy functionals, on which Theorems~\ref{t-main1} and \ref{t-main2} are based on.  In Section \ref{s4} we  give two examples which illustrate the effects that may happen in the ``extremely heavy-tailed'' case, i.e. when condition \eqref{m2r} (see below)  fails. In Appendix we give some supplementary statements: the necessary and sufficient condition for the integral (\ref{yt}) to be well defined,  and the condition for the respective distribution to possess a density.

\section{Settings and the main result}\label{s2}
Let $Z_t$, $t\in \Re$,  be a real-valued (two-sided) \emph{L\'evy process}  with
\emph{the characteristic exponent} $\psi$, which means  that $Z$ has stationary  independent increments, and the characteristic function of an increment is given by
\begin{equation}
E e^{i z (Z_t-Z_s)} =e^{(t-s)\psi(z)}, \quad t>s. \label{char}
\end{equation}
The characteristic exponent  $\psi$  admits the
L\'evy-Khinchin representation
\begin{equation}\label{LKh}
\psi(z)=iaz - b z^2+\int_{\real}
 \left( e^{iuz}-1- iz u 1_{\{|u|\leq 1\}} \right)\mu(du),
\end{equation}
where $a\in\real$, $b\geq 0,$ and $\mu(\cdot)$ is \emph{a L\'evy
measure}, i.e. $\int_{\real} (1\wedge u^2)\mu(du)<\infty$. To  exclude  the trivial cases, we assume that $b=0$ and $\mu(\Re)>0$; that is,  $Z$ does not contain a diffusion part, and contains a non-trivial jump part. To simplify the notation we also assume without loss of generality that  $Z_0=0$ and $a=0$.

We define the integral (\ref{yt}) as a limit in probability of the respective integral sums, see \cite[Section 2]{RR89}. When $H\not=1/2$, the necessary and sufficient condition  for this integral to be well defined is
    \begin{equation}
    \int_{|u|\geq 1} |u|^{2/(3-2H)}  \mu(du)<\infty, \label{esint}
    \end{equation}
see Proposition \ref{p21} in Appendix~I. Furthermore, it will be shown in Proposition~\ref{p22} (see Appendix~II) that under the   same conditions  and
 our standing  assumption
\begin{equation}
\mu(\Re)>0,\label{mu0}
 \end{equation}
the integral (\ref{yt}) possesses for any $t\not=0$  a distribution density, which we denote by  $p_t(x)$, and, moreover, $p_t\in C_b^\infty(\real)$. Note that in the L\'evy case, i.e. for $H=1/2$, available  sufficient conditions  for the existence of the density $p_t\in C_b^\infty(\real)$ are much stronger, see, for example,  \cite{HW42}  for  the
\emph{ Hartman-Wintner condition}.

An important feature of the process $Z_t^H$ is that one can explicitly write its characteristic function $\phi(t,z):=Ee^{izZ^H_t}$ (cf.  \cite[Theorem 2.7]{RR89}):
\begin{equation}\label{chint}
 \phi(t,z)=e^{\Psi(t,-z)},
 \end{equation}
 where
 \begin{equation}
 \Psi(t,z)=
 \int_\Re \int_{\Re}\left(e^{-izf(t,s)u}-1+izf(t,s)u 1_{|u|\leq 1}\right)\,\mu(du)ds, \quad z\in\real, \quad t>0.\label{psi1}
\end{equation}
Observe, that if the measure $\mu$ possesses exponential moments, the function $\Psi(t,z)$ can be  extended with respect to $z$ to the complex plane. Moreover, one can see (cf. Section~\ref{3.2}) under the assumptions that $\mu(\real_+)>0$,  the function
$$
H(x,z):= izx+\Psi(1,z)
$$
has  a unique critical point $i\xi(x)$ on the line $i\real$. Put
\begin{equation}
D(x):=H(x,i\xi(x)),\quad
K(x):=\frac{\partial^2}{\partial \xi^2} H(x,i\xi)\Big|_{\xi=\xi(x)},\label{DH}
\end{equation}
and
$$
M_k(\xi):= \int_\real u^k e^{\xi u}\mu(du), \quad k\geq 2, \quad \xi\in \Re.
$$

Fix $t_0>0$. In what follows, we write $f\ll g,$  if $f/g\to 0,$ and $f\sim g$, if $f/g\to 1$.

\begin{theorem}\label{t-main1} Let $Z_t^H$, $1/2< H<1$, $t\geq t_0$, be a FLM defined by (\ref{yt}), where $Z_t$ is  a L\'evy process with the  associate L\'evy measure $\mu$.
Suppose that the conditions below hold true:

1) $\mu(\real^+)>0$;

2) for all $C\in \real$
\begin{equation}
\int_{|y|\geq 1} e^{Cy} \mu(dy)<\infty; \label{exp}
\end{equation}

3) $\exists \gamma\in (0,1)$ such that $M_4(\xi)\ll M_2^2(\gamma \xi)$ as $\xi \to \infty$;

4) $\ln\Big(\frac{M_4(\xi)}{M_2(\xi)}\vee 1\Big) +\ln \ln
M_2(\xi)\ll \xi$, $\xi\to+\infty$.

Then the distribution density $p_t(x)$ of $Z_t^H$ exists, $p_t\in C_b^\infty$, and satisfies  the asymptotic relation
\begin{equation}
p_t(x)\sim \frac{1}{\sqrt{2\pi t^{2H}  K(xt^{-H-1/2})}} e^{tD(xt^{-H-1/2})}, \quad t+x\to\infty,\quad (t,x)\in [t_0,\infty)\times \real_+,\label{ptx1}
\end{equation}
where $D$ and $K$ are defined in  (\ref{DH}).
\end{theorem}
To formulate the result for $0<H<1/2$ we need a bit more notation.
Let $f(s):=f(1,s)$, see (\ref{f1}) for the definition of $f(t,s)$. Observe that $f(s)$  is monotone decreasing on $(-\infty,0)$,  monotone increasing on $(0,1]$, and  the range of $f$ restricted, respectively,  to $(-\infty, 0)$ and  $(0,1]$,  is  $(-\infty,0)$ and   $\big[\frac{1}{\Gamma(H+1/2)},+\infty \big)$.    In addition, the derivative $f'(s)$ is well defined and is continuous on $(-\infty,0)$ and $(0,1)$.   Hence, we can put
\begin{equation}
\ell (y):=
\begin{cases}
(f^{-1})'(y)=\frac{1}{f'(f^{-1}(y))}, \quad  &y\in (-\infty, 0)\cup \Big[\frac{1}{\Gamma(H+1/2)},+\infty\Big),
\\
0,\quad  &\hbox{otherwise}.
\end{cases} \label{ell}
\end{equation}
Note that $\ell(y)$ is non-negative if $y\geq 0$, and is negative otherwise. Define
\begin{equation}
\mathfrak{m}(r):=  \int_{-\infty}^{\infty} \frac{1}{y} \ell\left(\frac{r}{y}\right)\mu(dy), \quad r>0. \label{m}
\end{equation}

Recall that (see Definition 4 in \cite{Kl89}) a function $g:\Re_+\to \Re_+$  belongs to the class $\Ld$ of \emph{sub-exponential densities}, if $g(x)>0$ for large enough positive $x$, and
\begin{equation}
 \lim_{x\to+\infty} \frac{(g*g)(x)}{g(x)}=2, \quad \text{and} \quad \lim_{x\to+\infty} \frac{g(x-y)}{g(x)}=1 \text{ for any $y
\in \Re$,} \label{m3r}
\end{equation}
where $*$ is a usual definition for the convolution. Fix $t>0$.
\begin{theorem}\label{t-main2}
Let $Z_t^H$, $0< H<1/2$,  be a FLM defined by (\ref{yt}), where $Z_t$ is  a L\'evy process with the  associate L\'evy measure $\mu$.
Suppose that (\ref{esint}) holds true and
\begin{equation}
    \mu(\Re_-)>0. \label{mu1}
    \end{equation}
Then for every $t>0$ the value of FLM $Z_t^H$ defined by (\ref{yt}) possesses a  probability density $p_t\in C_b^\infty(\real),$ which satisfies the following:

i)  if
\begin{equation*}
 \int_{|u|\geq 1} |u|^{2/(1-2H)}  \mu(du)=\infty
\end{equation*}
     and $\mathfrak{m}\in \Ld$, then  for all $t>0$
    \begin{equation}
    p_t(x)\sim t^{3/2-H} \mathfrak{m}(t^{1/2-H}x), \quad  x\to +\infty;  \label{main-es2}
    \end{equation}

ii)  if
\begin{equation}
 \int_{|u|\geq 1} |u|^{2/(1-2H)}  \mu(du)<\infty,\label{m2r}
\end{equation}
then
\begin{equation}
    p_t(x)\sim c_H \left(\int_\Re |u|^{2/(1-2H)} \mu(du)\right)x^{-{(3-2H)/(1-2H)}},\quad   x\to +\infty,  \label{main-es3}
    \end{equation}
where
 \begin{equation}
 \label{CH}
c_H=\frac{2}{1-2H}\left(\Gamma\Big(H+1/2\Big)\right)^{-2/(1-2H)}.
\end{equation}

\end{theorem}

\begin{remark}\label{rem1} a) Apparently,  relation (\ref{main-es2}) holds  true when both $x$ and $t$ tend to $+\infty$ in such a way that  $xt^{-H-1/2}\to +\infty$. To prove such an extension of Theorem \ref{t-main2}, one should have an extension of  \cite[Theorem 3.2]{Kl89} which applies to \emph{a family} of random sums with the variable distribution of the number of summands.

b)  Clearly, results similar to (\ref{main-es2}) and (\ref{main-es3}) can be formulated for $x\to -\infty$. In that case  one should assume $\mu(\Re_+)>0$ instead of $\mu(\Re_-)>0$.
\end{remark}

To illustrate  the crucial difference between the cases treated in Theorem~\ref{t-main1} and Theorem~\ref{t-main2},    consider  two particular examples from  \cite{KK11} which concern the case $H> 1/2$. First, let the L\'evy measure $\mu$ of the L\'evy noise $Z_t$ in (\ref{yt}) be supported in a bounded set. Then (see in \cite[Corollary 5.1 and Corollary 5.2]{KK11})  there exists a constant $c_*(\mu)$, defined in terms of  the L\'evy measure $\mu$ only, such that for any constants $c_1>c_*(\mu)$ and $c_2<c_*(\mu)$ there exists $y(c_1,c_2)$ such that for  $x>y(c_1, c_2)t^{H+1/2}$ we have
 \begin{equation}\label{answ721}
 p_t(x)\quad\begin{cases}\geq \exp\left(-{c_1x \over \Gamma(H+1/2)t^{H-1/2}} \ln \left(x\over t^{H+1/2}\right)\right), \\  \leq \exp\left(-{c_2x\over \Gamma(H+1/2)t^{H-1/2}} \ln \left(x\over t^{H+1/2}\right)\right).
 \end{cases}
\end{equation}
Similar statement is available as well when the tails of the L\'evy measure admit the following super-exponential estimates: for $u$ large enough
\begin{equation}\label{mu11}
{1\over Q(u)}e^{-b u^\beta}\leq \mu([u,+\infty))\leq Q(u)e^{-b u^\beta},
\end{equation}
where $b>0$ and  $\beta>1$ are some constants,  and $Q$ is some polynomial. In  this   case, instead of (\ref{answ721}) we have  for any constants $c_1>c_*(\mu)$ and $c_2<c_*(\mu)$
\begin{equation}
\label{answ722}
 p_t(x)\quad
 \begin{cases}\geq
 \exp\left(-{c_1x \over \Gamma(H+1/2)t^{H-1/2}} \ln^{\beta-1\over \beta} \left(x\over t^{H+1/2}\right)\right)
 \\
 \leq \exp\left(-{c_2x\over \Gamma(H+1/2)t^{H-1/2}} \ln^{\beta-1\over \beta} \left(x\over t^{H+1/2}\right)\right),
 \end{cases}
\end{equation}
for  $x>y(c_1, c_2)t^{H+1/2}$ (again, $c_*(\mu)$ is defined in terms of  the L\'evy measure $\mu$ only).

Comparing (\ref{answ721}) and (\ref{answ722}), we see that  the asymptotic behaviour of the  tails of the L\'evy measure $\mu$ is substantially involved in the estimates for  $p_t(x)$.   The case $0<H<1/2$ is completely different.  In particular, if \eqref{m2r} holds true, then $p_t(x)$ satisfies (\ref{main-es3}),  where the right-hand side is even independent of $t$, which is an interesting and quite an unexpected fact. We also
 emphasize that under (\ref{m2r}) the polynomial ``shape'' of the expression in the right-hand side of (\ref{main-es3}) does not depend on $\mu$,  and the only impact of $\mu$ is represented by the multiplier  $\int_\Re |u|^{2/(1-2H)} \mu(du)$. This means that in the case $0<H<1/2$  the asymptotic behaviour of $p_t(x)$ ``mostly'' does not depend on $\mu$.  However, when $\mu$ is ``extremely heavy-tailed'', there still remains a possibility for the density $p_t(x)$ to be more sensitive with respect to both the L\'evy measure $\mu$ and the time parameter $t$. The dichotomy between the ``regular'' case (when (\ref{m2r}) holds) and the ``extremely heavy-tailed'' case (when (\ref{m2r}) fails) is illustrated in  Section \ref{s4} below. Such a dichotomy can be informally explained by the competition between the impacts of the kernel $f(t,s)$ on one hand, and of the measure $\mu$ on the other hand.

\section{Proofs} \label{s3}

\subsection{General theorem} Before we proceed to  the proofs, we formulate a central analytical result on the behaviour of the inverse Fourier transform for a certain class of functions. This result plays the key role in   the proofs of Theorems~\ref{t-main1} and \ref{t-main2}.

 Let  $I\subset \Re$ be some interval and $\mathbb{T}$ be some set of parameters.  Consider a function  $f:\mathbb{T}\times I\to \Re$, a family of subsets $\mathcal{C}(t,s)\subset \Re$, $t\in \mathbb{T}$, $s\in I$, and a L\'evy measure $\mu$ such that
  \begin{equation}
 \int_I f^2(t,s)\, ds<\infty, \quad t\in \mathbb{T},\label{f11}
 \end{equation}
 \begin{equation}
\int_I \int_{\mathcal{C}(t,s)} \left(  |f(t,s)u|^2\wedge 1\right) \mu(du)ds<\infty,\label{exi1A}
\end{equation}
 \begin{equation}
   \int_I \left|\int_{\mathcal{C}(t,s)} \left( f(t,s)u 1_{|f(t,s)u|\leq 1} - f(t,s)u 1_{|u|\leq 1}\right) \mu(du)\right| ds<\infty.\label{exi2A}
    \end{equation}
Then the following function is well defined:
 \begin{equation}
 \Psi(t,z)= \int_I \int_{\mathcal{C}(t,s) } \big( e^{-izf(t,s)u} -1+izf(t,s)u 1_{|u|\leq 1} \big) \mu(du)ds, \quad z\in \real. \label{phi1}
 \end{equation}

   Our aim is to investigate the asymptotic behaviour of the function (provided it exists)
   \begin{equation}
  q_t(x)=  (2\pi)^{-1} \int_\real e^{ixz}\phi(t,z)dz=(2\pi)^{-1} \int_\real e^{-ixz+ \Psi(t,z)}dz, \quad t>0, \quad x\in \real, \label{q1}
   \end{equation}
as $(t,x)$ tend to infinity in some appropriate regions. Clearly, when $\mathcal{C}(t,s) \equiv \real$, the function $q_t(x)$ is nothing else but the distribution density of the L\'evy functional
\begin{equation}
Y_t= \int_I f(t,s) dZ_s,
\end{equation}
where $Z_s$ is the L\'evy process associated with measure $\mu$, without a drift and a Gaussian component.

Assume in addition that for some $\lambda>0$ we have
\begin{equation}\label{H1}
f(t,s)u\leq \lambda, \quad t\in \mathbb{T}, \quad s\in I, \quad u\in \mathcal{C}(t,s).
    \end{equation}
Then it can be shown that the function $\Psi(t, \cdot)$ defined in  (\ref{phi1}) can be extended to the half-plane  $\mathbb{C}_+:=\{z\in \mathbb{C}:\mathrm{Im}\, z\geq 0\}$, and respective extension (we denote it by the same letter $\Psi$) is continuous on $\mathbb{C}_+$ and analytical in the inner part of this half-plane.

Consider  the function
\begin{equation}
    H(t,x,z):=izx+\Psi(t,z), \quad  z\in \mathbb{C}_+, \label{H}
    \end{equation}
and observe that
$$
 \frac{\partial}{\partial \xi} H(t,x,i\xi)=-x+ \int_I \int_{\mathcal{C}(t,s) } u f(t,s)\big( e^{\xi f(t,s)u} - 1_{|u|\leq 1} \big) \mu(du)ds\to \infty, \quad \xi\to +\infty,
 $$
 provided that
\begin{equation}\label{non-zero}
(\mu\times \mathrm{Leb})\Big(\{(u,s): f(t,s)u>0\}\Big)>0.
\end{equation}
Furthermore, under the same condition $\frac{\partial^2}{\partial \xi^2} H(t,x,i\xi)>0$ for all $\xi\in \real$. Hence,
$$
(0, \infty)\ni\xi\mapsto \frac{\partial}{\partial \xi} H(t,x,i\xi)
$$
is a continuous strictly increasing function with the range $(x_t-x, \infty)$, where
\begin{equation}\label{x_*}
x_t=\lim_{\xi\to 0+}\Psi(t, i\xi)=\int_I \int_{\mathcal{C}(t,s) } u f(t,s)1_{|u|>1}  \mu(du)ds.
\end{equation}
 Note that by the above conditions on $f,\mu$ and $\mathcal{C}(t,s)$ the value $x_t$ may equal $-\infty$, but  is less than $+\infty$. Then for any $x>x_t$ there exists unique solution $\xi(t,x)$ to the equation
\begin{equation}
\frac{\partial}{\partial \xi} H(t,x,i\xi)=0.  \label{extr}
\end{equation}

To formulate the result we need some extra notation:
    \begin{equation}
    \mathcal{M}_k(t,\xi):=\frac{ \partial^k}{\partial \xi^k} \Psi(t,i\xi), \quad k\geq 1,
    \end{equation}
    \begin{equation}
     \mathcal{D}(t,x):=H(t,x,i\xi(t,x)), \quad \mathcal{K}(t,x):=\mathcal{M}_2(t,\xi(t,x)),
    \end{equation}
  and
$$
\Theta(t,z,B):=  \int_I \int_{\{u:\, f(t,s)u\in B,\,
                            u\in \mathcal{C}(t,s)\}}  \Big(1-\cos (f(t,s) zu)\Big)\mu(du)ds.
$$
  Consider a set  $\mathcal{A}\subset \{(t,x):  t\in \mathbb{T}, x>x_t\}\subset \mathbb{T}\times \real$, and
define
$$
\mathcal{T}:=\{t: \exists x\in (x_t, \infty), (t,x)\in \mathcal{A}\},
\quad \mathcal{B}:=\{(t,\xi): \exists (t,x)\in \mathcal{A},
(t,\xi)=(t, \xi(t,x))\}.
$$
Finally, suppose that the function $\theta:\mathcal{T}\to (0,+\infty)$ is  bounded away from zero on $\mathbb{T}$, and the function
$\chi:\mathcal{T}\to (0,+\infty)$ is  bounded away
from zero on every set  $\{t:\theta(t)\leq c\}$, $c>0$.

\begin{theorem}\label{td1}   Assume the following.

\medskip

\textbf{H1}  Conditions (\ref{f11}) - (\ref{exi2A}), (\ref{H1}), and (\ref{non-zero}) hold true.

\textbf{H2} $\Mf_{4}(t,\xi)\ll \Mf_{2}^2(t,\xi)$, $\theta(t)+\xi\to \infty$, $(t,\xi)\in\mathcal{B}$.

\textbf{H3} For $\theta(t)+\xi\to \infty$,  $(t,\xi)\in\mathcal{B}$.
\begin{align*}
\ln\left(\Big(\chi^{-2}(t)\frac{\Mf_{4}(t,\xi)}{\Mf_{2}(t,\xi)}\Big)
\vee 1\right)&+\ln\left(\Big(  \ln  \left((1\vee \chi^{-1}(t))\Mf_{2}(t,\xi)\right)
\Big)\vee1\right)\\& \ll \ln \theta(t)
+\chi(t)\xi.
\end{align*}

\textbf{H4} There exist $R>0$ and  $\delta>0$ such that
\begin{equation}\label{h5}
 \Theta(t,z,\Re_+)\geq (1+\delta)\ln (\chi(t) |z|),\quad t\in\mathcal{T}, \quad  |z|>R.
\end{equation}

\textbf{H5} There exists $r>0$ such that for every $\epsilon>0$,
$$
\inf_{|z|>\epsilon}\Theta(t,z,[r\chi(t),+\infty))\geq   c \theta(t)\Big((\epsilon\chi(t))^2\wedge 1\Big),\quad t\in\mathcal{T},\quad  c>0.
$$

Then the function $q_t(x)$ given by (\ref{q1}) is well-defined, and  satisfies
\begin{equation}
q_t(x)\sim \frac{1}{\sqrt{2\pi \mathcal{K}(t,x)}} e^{\mathcal{D}(t,x)}, \quad\theta(t)+x\to\infty, \quad (t,x)\in \mathcal{A}.\label{ttp}
    \end{equation}
\end{theorem}
Up to some straightforward and purely technical modifications, the proof of  Theorem~\ref{td1} coincides with the proof of  \cite[Theorem 2.1]{KK11}, and therefore is omitted. Here we only remark that the proof is based on an appropriate modification of the \emph{saddle point method}, see  \cite{Co65} for details.

\subsection{Outline of the proofs}

One can prove Theorem~\ref{t-main1} using a simplified version of  Theorem~\ref{td1} with $\mathcal{C}(t,s)\equiv \real$, and  the scaling property of the function $f(t,s)$:
\begin{equation}
    f(t,s) =|t|^{H-1/2} f\left(\frac{s}{t}\right),  \label{f}
    \end{equation}
   where
   $$
   f(s)=f(1,s)=  \frac{1}{\Gamma(H+1/2)} \left[(1-s)_+^{H-1/2}-(-s)^{H-1/2}_+\right], \quad s\in \Re;
   $$
 see \cite{KK11} for details.

Let us turn now to the proof of Theorem~\ref{t-main2}. To make the proof of Theorem~\ref{t-main2}  more transparent, we first sketch its main idea. In particular, we show how  Theorem~\ref{td1} applies in the situation when the function $f(t,s)$ is unbounded.

According to  \cite[Theorem 2.7]{RR89}, the characteristic function $\phi(t,z)$ (cf. (\ref{chint})) can be decomposed for any fixed $\lambda>0$ as
    $$
    \phi(t,z)=\phi_1(t,z)\phi_2(t,z)=e^{\psi_1(t,z)}e^{\psi_2(t,z)},
    $$
where
\begin{equation}
 \psi_1(t,z):=\int_{-\infty}^t\int_{\{u:\, uf(s/t)\leq \lambda\}}\left(e^{izf(t,s)u}-1-izf(t,s)u 1_{|u|\leq 1}\right)\,\mu(du)ds,\label{phi11}
    \end{equation}
    \begin{equation}
\psi_2(t,z):=\int_{-\infty}^t\int_{\{u:\, uf(s/t)> \lambda\}}\left(e^{izf(t,s)u}-1\right)\,\mu(du)ds+ iza(t),\label{phi2}
    \end{equation}
    \begin{equation}\begin{split}
     a(t):= \int_{-\infty}^t\int_{\{u:\, uf(s/t)> \lambda\}}f(t,s) u 1_{|u|\leq 1} \mu(du)ds=t^{H+1/2} a(1). \label{at}
    \end{split}
      \end{equation}
In the last identity we  have used the  scaling property (\ref{f}) of the kernel $f$.

The function  $|\phi_1(t,z)|$ is integrable with respect to  $z$ for any $t>0$, see Remark \ref{rb2} in the Appendix~II below. Then  there exists the distribution density
    \begin{equation}
    \tilde{p}_t(x):=\frac{1}{2\pi} \int_\Re  e^{-izx} \phi_1(t,z)dz, \label{tp}
    \end{equation}
and thus the required density $p_t(x)$ can be written as the convolution
    \begin{equation}
    p_t(x)=(\tilde{p}_t* P_t)(x),\label{p1}
    \end{equation}
where $P_t(dy)$ is the probability measure corresponding to  the  characteristic function $\phi_2(t,z)$. Define the measure $M_t(dy)$ by the relation
\begin{equation}
\int_{\Re}g(y) M_t(dy)=\int_{-\infty}^t\int_{\{u:\, uf(s/t)> \lambda\}}g(f(t,s)u)\,\mu(du)ds,\label{M}
\end{equation}
where $g$ is an arbitrary measurable and bounded function. Then, up to the shift by $a(t)$, the measure $P_t(dy)$ is equal to the distribution of
the compound Poisson random variable  with the intensity of the Poisson part equal to $M_t(dy)$. In other words,
\begin{equation}
    P_t(dy) = \delta_{a(t)}(dy)* \left( e^{-M_t(\Re)}\delta_{0}(dy) +   e^{-M_t(\Re)} \sum_{k=1}^\infty \frac{1 }{k!}   M_t^{*k} (dy)\right), \label{p2}
    \end{equation}
where  $M_t^{*k}(dy)$ is the $k$-fold convolution of $M_t(dy)$.

It follows from the scaling property (\ref{f}) that $M_t(\Re)=t\Lambda$ with
\begin{equation}
\Lambda= \int_{-\infty}^1 \int_{\{u: \, uf(s)> \lambda\}} \mu(du). \label{ML}
\end{equation}
Furthermore, we prove that the measure $M_t(dy)$ is absolutely continuous with respect to the Lebesgue measure with the density
\begin{equation}
\label{mt1}m_t(x)=t^{3/2-H} \mathfrak{m}(t^{1/2-H}x)1_{\{x>\lambda t^{H-1/2}\}},
\end{equation}
 where the function $\mathfrak{m}$ is defined by (\ref{m}).  Then (\ref{p1}) can be written in the form
   \begin{equation}
    p_t(x)=e^{-\Lambda t} \tilde{p}_t(x-a(t)) + \int_\Re \rho_t(x-y)\tilde{p}_t(y-a(t))dy,\label{p3}
    \end{equation}
where
\begin{equation}
\rho_t(x):=e^{-t\Lambda} \sum_{k=1}^\infty \frac{m_t^{*k}(x)}{k!}.\label{rho1}
\end{equation}

 Clearly, $\rho_t$ is the density of a random sum with the distribution of one term represented by $m_t$.  Suppose that the function $\mathfrak{m}$ is sub-exponential. Then it follows from  \cite[Theorem 3.2]{Kl89} that the density $\rho_t$ is sub-exponential as well, and
\begin{equation}\label{rho}
\rho_t(x)\sim \left(e^{-t\Lambda} \sum_{k=1}^\infty{k\over k!}(t\Lambda)^{k-1}\right)m_t(x)=m_t(x), \quad x\to +\infty,
\end{equation}
where we used that $\int_0^\infty m_t(x)\, dx=M_t(\Re)=t\Lambda$.

To estimate $\tilde{p}_t(x)$ we apply Theorem~\ref{td1}.   Namely, in Proposition~\ref{pr1} below we show that for a given $\varepsilon>0$ there exists $y(\varepsilon)>0$ such that
    \begin{equation}\begin{aligned}
   \exp\left(-\frac{(1+\varepsilon) x }{\lambda t^{H-1/2}}\ln \frac{x}{t^{H+1/2}}\right)\leq  \tilde{p}_t(x)\leq &\exp\left(-\frac{(1-\varepsilon) x }{\lambda t^{H-1/2}}\ln \frac{x}{t^{H+1/2}}\right),\\& \hskip 2.5cm xt^{-H-{1/2}}\geq y(\eps).  \end{aligned}\label{es1}
     \end{equation}
Since a sub-exponential function  decays slower than any exponential function (cf. \cite{Kl89}), the term $m_t$ dominates both $\tilde{p}_t(x)$ and the integral term in (\ref{p3}). In such a way,  (\ref{p3}), (\ref{rho}) and (\ref{es1}) provide the required relation (\ref{main-es2}).

\medskip

Let us summarize the idea explained above.   The distribution of $Z_t^H$ is decomposed in two parts. For one part, the distribution density is controlled by means of the respective version of the saddle point method,  while for the other part the distribution can be evaluated in the form of the series of convolution powers with the explicitly given law of the first summand. Similarly to Theorem 2.1 in \cite{KK11}, Theorem~\ref{td1}   provides a flexible  version of the saddle point method, which is applicable to a wide variety of integrals of the form (\ref{q1}). Thus one can expect that the approach  presented above  can be extended to other processes of the form  (\ref{yt}) with unbounded kernels $f(t,s)$. To keep the exposition reasonably tight, in this paper we do not investigate this possibility in the whole generality, and restrict our considerations to the important particular case of the  FLM with $0<H<1/2$.

\subsection{Properties of $\tilde{p}_t(x)$}\label{3.2}
\begin{proposition}\label{pr1} Under  (\ref{mu1}), for a given $\varepsilon>0$ there exists $y(\varepsilon)>0$ such that $\tilde{p}_t(x)$ satisfies (\ref{es1}).
 \end{proposition}
\begin{remark}
Note that in the above Proposition we do not assume that $t>0$ is fixed.
\end{remark}
\begin{proof}  We use Theorem~\ref{td1} with $\theta(t)=t$, $\chi(t)=t^{H-\frac{1}{2}}$, $\mathbb{T}=[t_0,\infty)$,  $I=(-\infty, t]$, $\mathcal{C}(t,s)=\{u:f(t,s)u\leq \lambda\}$,   and  $\mathcal{A}=\{ (t,x)\subset [t_0,\infty)\times \Re_+:\,\, x t^{-H-1/2} \geq c\}$. Here $t_0$, $c$ and $\lambda$ are some positive constants.  Then condition \textbf{H1} is satisfied: (\ref{H1}) holds true by the construction, (\ref{non-zero}) holds true thanks to (\ref{mu1}), and  (\ref{f11}) - (\ref{exi2A}) can be proved using the same estimates as in the proof of Proposition \ref{p21} in Appendix~I (we omit the details).

 Recall that $f$ possesses   the self-similarity property (\ref{f}).
Then
    \begin{equation}
    \Mf_k(t,\zeta)=\chi^k(t) t \Mf_k(\chi(t)\zeta),\label{mk-scal}
    \end{equation}
where
\begin{equation*}
\begin{split}
\Mf_k (\zeta)&:= \frac{\partial^k}{\partial \zeta^k} H(1,x,i\zeta(1,x))=
\begin{cases}
\int_{-\infty}^\lambda u (e^{u\zeta}-1)N(du),& k=1,\\
\int_{-\infty}^\lambda u^k e^{u \zeta}N(du),& k\geq 2.
\end{cases}
\end{split}
\end{equation*}
Here  $N(du):= \int_{-\infty}^1 \tilde{\mu}_s(du)ds$, and $\tilde{\mu}_s(du)$ is the image measure of $\mu(du)$ under the mapping $u\mapsto f(s)u$.
The choice of $\lambda$ above can be made in such a way that every segment $(\lambda-\eps, \lambda)$ has a positive measure $N$. Then it can be shown (e.g.,    \cite[Example~3.1]{KK11}) that for any $\varepsilon>0$
 \begin{equation}\label{star}
e^{\zeta(\lambda-\varepsilon)} \ll \Mf_k(\zeta), \quad \Mf_k(\zeta)-\lambda^k N(\{\lambda\})e^{ \zeta \lambda}\ll e^{\zeta \lambda}, \quad \zeta\to+\infty.
\end{equation}
Moreover,  applying the Laplace method we get
\begin{equation}
\Mf_k(\zeta)\sim \lambda^k \Mf_0(\zeta),\quad \zeta \to \infty, \label{mk-m0}
\end{equation}
where
$$
\Mf_0(\zeta)=\int_{-\infty}^\lambda\Big( e^{\zeta u }-1-\zeta u \Big) N(du).
$$
Note that the solution $\xi(t,x)$ to (\ref{extr}) satisfies
\begin{equation}
\xi(t,x)=\chi^{-1}(t) \zeta(xt^{-H-1/2}), \quad \text{where}\quad \zeta(x):=\xi(1,x).\label{xz}
\end{equation}
 Since $\zeta(x)$ is the solution to $\Mf_1(\zeta(x))=x$, we have $\zeta(x)\to\infty$ as $x\to\infty$.
 Then by \eqref{xz} and the definition of $\mathcal{A}$     we have $\chi(t)\xi\to\infty$ as
$t+\xi\to\infty$, $(t,\xi) \in \mathcal{B}$, implying
$$
\Mf_0(\chi(t)\xi)>0, \quad \text{as $t+\xi\to\infty$, $(t,\xi) \in \mathcal{B}$}.
$$
   Therefore, by  (\ref{mk-scal}) and (\ref{mk-m0}) we have \textbf{H2}:
$$
\frac{\Mf_4(t,\xi)}{\Mf_2^2(t,\xi)}\sim \frac{1}{t \Mf_0(\chi(t)\xi)}\ll 1, \quad t+\xi\to\infty, \quad (t,\xi) \in \mathcal{B}.
$$
Analogously, we have
$$
\chi^{-2}(t) \frac{\Mf_4(t,\xi)}{\Mf_2(t,\xi)}\sim \lambda^2,
$$
$$
\ln\left(\Big(  \ln  \left((1\vee \chi^{-1}(t))\Mf_{2}(t,\xi)\right)
\Big)\vee1\right) \leq C(\ln\ln t+\ln(\chi(t)\xi))\ll \ln t + \chi(t)\xi,
$$
as $t+\xi\to\infty$, $(t,\xi) \in \mathcal{B}$, which provides \textbf{H3}.

 To show \textbf{H4}, take $b<0$  and $\varepsilon>0$ such that
$0<\mu([\lambda/ f(b), -\varepsilon])<\infty$. Then by (\ref{exh2}) (see Appendix~II) we have for $|z|\geq R$ with some $R$ large enough
    \begin{align*}
    \Theta(t,z,\Re_+) &\geq  t \int_{-\infty}^0 \int_{\{u: \, 0<f(s)u\leq \lambda\}} (1-\cos (\chi(t) zuf(s)))\mu(du)ds
    \\&
    \geq t_0\int_{-\infty}^b \int_{\lambda/f(b)}^{-\varepsilon}  (1-\cos (\chi(t) zuf(s)))\mu(du)ds
    \\&
    \geq c t_0\mu([\lambda/f(b),-\varepsilon]) \ln (\chi(t) \varepsilon|z|).
    \end{align*}
Since  we can chose in (\ref{exh2}) $c>0$ arbitrary large, condition \textbf{H4} holds true. Finally, estimate (\ref{exh1}) (see Appendix~II) provides \textbf{H5}:  since $\mu(\Re_-)>0$,  there exists $(a,b)\subset (-\infty,0)$, $q>0$, such that $0<\mu([q\lambda/f(b), \lambda/f(b)])<\infty$ (note that $f(b)<0$), and
    \begin{align*}
    \inf_{|z|\geq c} \Theta(t,z,[\chi(t) q, \infty))&\geq t  \inf_{|z|\geq c}\int_a^b \int_{q\lambda/f(s)}^{\lambda/f(s)}
    (1-\cos (\chi(t) zuf(s)))\mu(du)ds
    \\&
    \geq  t  \inf_{|z|\geq c} \int_{q\lambda/f(b)}^{\lambda/f(b)} ((z\chi(t) u)^2 \wedge 1) \mu(du)
    \\&
    \geq c_1  t ((\chi(t)c)^2\wedge 1).
    \end{align*}
Thus, all conditions of Theorem~\ref{td1} are satisfied, and therefore (\ref{ttp}) holds true.

By (\ref{star}),
$$
 \ln\Mf_1(\zeta)\sim  \lambda \zeta, \quad \zeta\to+\infty.
$$
Since   $\zeta(x)$ (cf. (\ref{xz}))  is the solution to $\Mf_1(\zeta(x))=x$, this means that
$$
\zeta(x)\sim {\ln x\over \lambda}, \quad x\to +\infty.
$$
Denote  $\mathcal{D}(x) := H(1,x,i\xi(1,x))$, then using \eqref{mk-m0} with $k=1,2$ we get from the previous relation that
    \begin{equation}
      \mathcal{D}(x)\sim -{x\ln x\over \lambda}, \quad \Mf_2(\zeta(x))\sim \lambda x,\quad x\to +\infty.\label{toin}
    \end{equation}
From (\ref{ttp}) and (\ref{toin}) we deduce that for a given $\varepsilon>0$ there exists $y(\varepsilon)>0$ such that  (\ref{es1}) holds true.
\end{proof}

\subsection{Properties of $M_t(dx)$ and the completion of the proof}

\begin{lemma}\label{mt}
For every  $t> 0$ we have $M_t(dx)=m_t(x)dx$ with
$m_t$ defined by (\ref{mt1}).
\end{lemma}
\begin{proof} Let $g$ be an arbitrary bounded measurable function.  Using the scaling property (\ref{f}) of the kernel $f(t,s)$ we can rewrite \eqref{M} as
 \begin{align*}
    \int_{-\infty}^\infty& g(y)M_t(dy)= \int_{-\infty}^t \int_{\{ yf(s/t)> \lambda\}} g( t^{H-1/2} f(s/t)y) \mu(dy)ds
    \\&
    =t\int_{-\infty}^0 \int_{\{y<-\lambda/f(s)\}} g_t (f(s)y) \mu(dy)ds+ t\int_0^1\int_{\{y>\lambda/f(s)\}} g_t (f(s)y) \mu(dy)ds
     \\&
     =:I_{-}+I_{+}.
    \end{align*}
Here  $g_t(y):= g( t^{H-1/2}y)$, and recall that the function  $f(s)=f(1,s)$ is monotone on $(-\infty, 0)$ and $(0,1)$, in particular,   $f$ is positive on $(0,1)$ and  negative on $(-\infty, 0)$. Let us transform the integrals $I_{+}$ and $I_{-}$ separately.

Recall that the range of the restriction of $f$ to $(0,1)$ equals $\Big[{1\over \Gamma(H+1/2)},+\infty \Big)$. Then, making the change of variables $\tau= 1/f(s)$, we get
   \begin{align*}
 I_{+}&=  t\int_{0}^{\Gamma(H+1/2)} \left(\int_{\{y>\lambda\tau\}}   g_t\left(\frac{y}{\tau}\right) \mu(dy)\right)
    \frac{d\tau}{\tau^2 f'(f^{-1}(\frac{1}{\tau}))}\\&
    =t\int_{0}^{+\infty} \left(\int_{\{y>\lambda\tau\}}   g_t\left(\frac{y}{\tau}\right) \mu(dy)\right)
   \ell\left(\frac{1}{\tau}\right) \frac{d\tau}{\tau^2}.
 \end{align*}
In the second identity we take into account that, according to  (\ref{ell}), the function  $\ell$ vanishes on $\Big(0, {1\over \Gamma(H+1/2)} \Big)$. Then the  further change of variables $r= y/\tau$ and the Fubini theorem give
$$
I_{+}= t\int_\lambda^\infty  g_t(r) \left(\int_0^\infty\frac{1}{y} \ell\left(\frac{r}{y}\right)\mu(dy)  \right)dr.
$$
Performing similar calculations, we get
$$
I_{-}= t\int_\lambda^\infty  g_t(r) \left(\int_{-\infty}^0\frac{1}{y} \ell\left(\frac{r}{y}\right)\mu(dy)  \right)dr.
$$
Adding the expressions for $I_+$ and $I_-$  we get
$$
 \int_{-\infty}^\infty g(y)M_t(dy)= t\int_\lambda^\infty  g( t^{H-1/2} r) \mathfrak{m}(r) dr=\int_{\lambda t^{H-1/2}} g(y) \Big[t^{3/2-H} \mathfrak{m}(t^{1/2-H}y)\Big]\, dy.
$$
 \end{proof}

  Let us summarize: we have (\ref{p3}) and (\ref{es1}); in addition, if  $\mathfrak{m}$ is sub-exponential, we have (\ref{rho}) by Theorem 3.2 in \cite{Kl89}. In such a way we obtain the proof of part i) of Theorem~\ref{t-main2}. The following lemma completes  the proof of the statement ii).

\begin{lemma}\label{mx2}
If  $\mu$ satisfies \eqref{m2r}, then
\begin{equation}\label{m1}
\mathfrak{m}(r)\sim c_H \left( \int_\real |u|^{2/(1-2H)} \mu(du)\right) r^{-(3-2H)/(1-2H)}, \quad r\to\infty,
\end{equation}
where the constant $c_H$ is defined in \eqref{CH}. In particular, $\mathfrak{m}\in \mathcal{L}d$.
\end{lemma}

\begin{proof}  Write $\mathfrak{m}=\mathfrak{m}_-+\mathfrak{m}_+$, where
\begin{equation}
\mathfrak{m}_-(r):=\int_{-\infty}^0 \frac{1}{y} \ell \Big(\frac{r}{y}\Big)\mu(dy),
\quad  \mathfrak{m}_+(r):=\int^{+\infty}_0 \frac{1}{y} \ell \Big(\frac{r}{y}\Big)\mu(dy).\label{m12-e10}
\end{equation}
On the positive half-axis the function $\ell$ can be calculated explicitly:
\begin{equation}
\ell (y)=   c_H y^{-(3-2H)/(1-2H)} 1_{\{y\geq 1/ \Gamma(H+1/2)\}}, \label{c1}
 \end{equation}
where  $c_H$  is given by (\ref{CH}). Then
\begin{align*}
\mathfrak{m}_+(r)&=  c_H r^{-(3-2H)/(1-2H)}\int_0^{r \Gamma(H+1/2)} y^{2/(1-2H)}\mu(dy)
\\&
\sim c_H \left(\int_0^{+\infty} y^{2/(1-2H)}\mu(dy)\right) r^{-(3-2H)/(1-2H)}, \quad r\to +\infty.
\end{align*}
Note that $2/(1-2H)>2$ and $\mu$ is a L\'evy measure (that is,  $\int_{|y|\leq 1}y^2\mu(dy)<\infty$), which together with \eqref{m2r}  implies that $\int_0^{+\infty} y^{2/(1-2H)}\mu(dy)<\infty$.

 On the negative half-axis one has
\begin{equation}
\ell(y)\sim
\begin{cases}
-c_H (-y)^{-(3-2H)/(1-2H)},& y\to-\infty,\\
-\hat c_H (-y)^{-(5-2H)/(3-2H)}, & y\to 0-,
\end{cases}\label{ell-as}
\end{equation}
with   $\hat c_H= \frac{2}{3-2H} \left(\frac{1-2H}{2\Gamma(H+1/2)}\right)^{2/(3-2H)}$. Take arbitrary $\eps>0$ and choose $a_\eps$, $b_\eps>0$ such that
$$
-\ell(y)\leq (\hat c_H+\eps) (-y)^{-(5-2H)/(3-2H)}, \quad (-y)\in (0,a_\eps), $$
$$  -\ell(y)\leq (c_H+\eps) (-y)^{-(3-2H)/(1-2H)}, \quad (-y)>b_\eps.
$$
Then
\begin{align*}
\mathfrak{m}_-(r)&=\left[\int_{-\infty}^{-r/a_\eps}+ \int_{-r/a_\eps}^{-r/b_\eps} + \int_{-r/b_\eps}^0\right]\frac{1}{y} \ell \Big(\frac{r}{y}\Big)\mu(dy)\\
& \leq (\hat c_H+\eps) r^{-(5-2H)/(3-2H)} \int_{-\infty}^{-r/a_\eps} (-y)^{2/(3-2H)}\mu(dy)\\&+\sup_{y\in [-b_\eps, -a_\eps]}\Big(-\ell (y)\Big)\int_{-r/a_\eps}^{-r/b_\eps}\left(-{1\over y}\right)\mu(dy)\\
&+ (c_H+\eps) r^{-(3-2H)/(1-2H)} \int_{-r/b_\eps}^0\!\! (-y)^{2/(1-2H)}\mu(dy)\\
& =I_1(r)+I_2(r)+I_3(r).
\end{align*}
By  condition \eqref{m2r}, one has
\begin{align*}
\int_{-\infty}^{-r/a_\eps}  (-y)^{2/(3-2H)}\mu(dy)& =\int_{-\infty}^{-r/a_\eps}  (-y)^{2/(1-2H)}(-y)^{-4/((1-2H)(3-2H))}\mu(dy)\\
&\leq \left(\tfrac{a_\epsilon}{r}\right)^{4/((1-2H)(3-2H))} \int_{-\infty}^{-r/a_\eps}  (-y)^{2/(1-2H)}\mu(dy)\\
&\leq c_1 r^{-4/((1-2H)(3-2H))},
\end{align*}
which implies
$$
r^{(3-2H)/(1-2H)}  I_1(r)\to 0, \quad r\to \infty.
$$
Further, since  by \eqref{m2r} we have
$$
 r^{2/(1-2H)}\mu((-\infty, -r])\leq \int_{-\infty}^{-r} (-y)^{2/(1-2H)}\mu(dy)\to 0, \quad r\to\infty,
 $$
then
\begin{align*}
r^{(3-2H)/(1-2H)} I_2(r) &\leq c_2  r^{2/(1-2H)} \mu\Big(\big(-\infty, -r/ b_\eps\big]\Big)\to 0, \quad r\to +\infty.
\end{align*}
 Thus,
$$
\limsup_{r\to +\infty} r^{(3-2H)/(1-2H)}\mathfrak{m}_-(r)\leq (c_H+\eps) \int_{-\infty}^0\!\! (-y)^{2/(1-2H)}\mu(dy).
$$
The same argument provides the desired lower bound for $\liminf_{r\to +\infty}$ with $c_H-\eps$ instead of $c_H+\eps$. Since $\eps$ is arbitrary, these two estimates lead to the relation
$$
\mathfrak{m}_-(r) \sim c_H \left(\int_{-\infty}^0 (-y)^{2/(1-2H)}\mu(dy)\right) r^{-(3-2H)/(1-2H)}, \quad r\to +\infty,
$$
which completes the proof.
\end{proof}

\section{Two examples: the ``extremely heavy-tailed'' case}\label{s4}
In this section we give two examples which illustrate the behaviour of $p_t(x)$  when  condition \eqref{m2r} fails. In this case  we say that the measure $\mu$  is ``extremely heavy-tailed''.

\begin{example}\label{exa1} Denote by $\mu_-(x)=\mu((-\infty, -x])$, $\mu_+(x)=\mu([x, +\infty))$, $x>0$,  the ``tails'' of the L\'evy measure $\mu$, and assume that $\mu_-$ and  $\mu_+$ are regularly varying at $+\infty$, that is, there exist $\alpha_\pm\in \Re$ and slowly varying functions  $L_\pm$, such that
$$
\mu_\pm(x)=x^{-\alpha_\pm} L_\pm(x),
$$
see, for example, \cite[Chapter VIII, \S 8]{Fe71}. We investigate the behaviour of the functions $\mathfrak{m}_-$ and $\mathfrak{m}_+$ introduced in the proof of Lemma \ref{mx2}.

We assume \begin{equation}\label{apm}
 \alpha_\pm\in \Big({2\over 3-2H}, {2\over 1-2H}\Big).
 \end{equation}
Note that condition $\alpha_{\pm}\geq 2/(3-2H)$ is necessary for  \eqref{esint} to hold true, and if $\alpha_\pm> 2/(1-2H)$ then \eqref{m2r} holds true and the required behaviour of  $\mathfrak{m}_-$ and $\mathfrak{m}_+$ is already described in Lemma \ref{mx2}. In order to  simplify the exposition,  we exclude from the consideration the critical values $\alpha_\pm=2/(3-2H)$ and $\alpha_\pm=2/(1-2H)$.

 The asymptotic behaviour of $\mathfrak{m}_+$ can be obtained almoststraightforwardly using the standard result on the behaviour of the integrals w.r.t. the measures with regularly varying tails, see   \cite[Chapter VIII, \S 9, Theorem 2]{Fe71}:
\begin{equation}\label{mplus}
\mathfrak{m}_+(r)\sim \left({\mu_{+}(r)\over r}\right)\int^{+\infty}_{0} \Phi(z)\, \mathfrak{p}_+(z)\, dz, \quad r\to \infty, \quad\hbox{with}\quad \mathfrak{p}_+(z)=\alpha_+z^{-\alpha_+-1}.
\end{equation}
  The investigation of the behaviour  of $\mathfrak{m}_-$ is slightly more complicated.  However, the  argument here is quite standard, and  therefore we just sketch it.

Write  $\mathfrak{m}_-$ in the form
$$
\mathfrak{m}_-(r)=r^{-1}\int^{0}_{-\infty} \Phi\left({y\over r}\right)\mu(dy),\quad \Phi(x):={1\over x}\ell\left({1\over x}\right).
$$
It follows from (\ref{ell-as}) that there exists a  constant $C$ such that  $\Phi(x)\leq C  (-x)^{2/(1-2H)} $ for $(-x)$ small enough  and $\Phi(x)\leq C (-x)^{2/(3-2H)} $ for $(-x)$ large enough. Then, by \cite[Chapter VIII, \S 9, Theorem 2]{Fe71}, (see also Problem~30 in \S10 of the same Chapter)  we have for $A$ small enough and $B$ large enough
\begin{align*}
\int^{-Br}_{-\infty}\Phi\left({y\over r}\right)\mu(dy)&\leq Cr^{-2/(3-2H)}\int^{-Br}_{-\infty}(-y)^{2/(3-2H)}\mu(dy)
\\&
 \leq C^1_{H,\alpha_-} B^{2/(3-2H)}\mu_-(Br),
\end{align*}
\begin{align*}
\int^{0}_{-Ar}\Phi\left({y\over r}\right)\mu(dy)&\leq C r^{-2/(1-2H)}\int_{-Ar}^{0}(-y)^{2/(1-2H)}\mu(dy)
\\&
 \leq C^2_{H,\alpha_-} A^{2/(1-2H)}\mu_-(Ar),
\end{align*}
with some explicitly given constant $C^i_{H, \alpha_-}\in (0, \infty)$, $i=1,2$. We have
$$
\limsup_{r\to +\infty} A^{2/(1-2H)}{\mu_-(Ar)\over \mu_-(r)}=A^{2/(1-2H)-\alpha_-} \limsup_{r\to +\infty}{L_-(Ar)\over L_-(r)}=A^{2/(1-2H)-\alpha_-}
$$
and, similarly,
$$
\limsup_{r\to +\infty} B^{2/(3-2H)} {\mu_-(Br)\over \mu_-(r)}=B^{2/(3-2H)-\alpha_-} .
$$
Then, by condition (\ref{apm}), for every $\eps>0$ one can choose $A$ and $B$ in such a way that
\begin{equation}\label{41}
\limsup_{r\to+\infty}{1\over \mu_-(r)}\left(\int_{-\infty}^0   \Phi\left({y\over r}\right)\mu(dy)- \int_{-Br}^{-Ar}\Phi\left({y\over r}\right)\mu(dy)\right)\leq \eps.
\end{equation}
Further, the function $\Phi$ is continuous and positive on $[-B,-A]$. Therefore, there exists a piece-wise constant function $\Phi_\eps$ such that
\begin{equation}\label{42}
(1-\eps)\Phi_\eps\leq \Phi\leq (1+\eps)\Phi_\eps.
\end{equation}
Clearly, one has for every segment $(a,b]\subset \real_-$
$$
\mu((ra,rb])\sim \mu_-(r)\int_a^b \mathfrak{p}_-(z)\, dz, \quad r\to \infty,\quad\hbox{with}\quad \mathfrak{p}_-(z)=\alpha_-(-z)^{-\alpha_--1}.
$$
Therefore,
$$
{1\over \mu_-(r)} \int_{-Br}^{-Ar}\Phi_\eps\left({y\over r}\right)\mu(dy)\to \int_{-B}^{-A} \Phi_\eps(z)\, \mathfrak{p}_-(z)\, dz, \quad r\to +\infty.
$$
Combined with (\ref{41}) and (\ref{42}), this gives
$$
\limsup_{r\to+\infty}{1\over \mu_-(r)} \int_{-\infty}^0  \Phi\left({y\over r}\right)\mu(dy)\leq \eps+{1+\eps\over 1-\eps} \int_{-\infty}^{0} \Phi(z)\, \mathfrak{p}_-(z)\, dz.
$$
One can write in the same fashion the lower bound for  $\liminf_{r\to+\infty}$ (we omit the calculation). Then, since $\eps>0$ is arbitrary, we  finally arrive at
$$
\mathfrak{m}_-(r)\sim \left({\mu_{-}(r)\over r}\right)\int_{-\infty}^{0} \Phi(z)\, \mathfrak{p}_-(z)\, dz, \quad r\to \infty.
$$
This and (\ref{mplus}) give that the function
$$
\mathfrak{m}(r)\sim{1\over r}\left(\mu_{-}(r)\int_{-\infty}^{0} \Phi(z)\, \mathfrak{p}_-(z)\, dz+ \mu_{+}(r)\int^{+\infty}_{0} \Phi(z)\, \mathfrak{p}_+(z)\, dz\right), \quad r\to +\infty,
$$
clearly belongs to the class $\Ld$,  and thus the statement i) of Theorem \ref{t-main2} holds true.

Consider, for instance, the ``$\alpha$-stable-like'' case
$$
\mu_-(r)\sim C_-r^{-\alpha}, \quad \mu_+(r)\sim C_+r^{-\alpha},\quad \hbox{with}\quad \alpha\in \Big({2\over 3-2H}, {2\over 1-2H}\Big).
$$
Then (\ref{main-es2}) and the above calculations give
\begin{equation}\label{main-es4}
p_t(x)\sim t^{1-\alpha(1/2-H)} x^{-\alpha-1}\int_\Re\Phi(z)\mu_{\alpha,C_-, C_+}(dz),\quad x\to +\infty
\end{equation}
with
$$
\mu_{\alpha,C_-, C_+}(dz)=\alpha|z|^{-\alpha-1}\Big(C_-1_{\Re_-}(z)+C_+1_{\Re_+}(z)\Big)\, dz.
$$
Note that the formal expression for $\mu_{\alpha,C_-, C_+}$ coincides with that for the L\'evy measure of an $\alpha$-stable distribution, although for $\alpha\in (2,2/(1-2H))$ an ``$\alpha$-stable distribution'' itself does not exist.

In contrast to (\ref{main-es3}),   formula (\ref{main-es4}) contains explicitly the time parameter $t$. In addition, the polynomial ``shape'' of the expression in the right-hand side of (\ref{main-es4}) depends on $\alpha$, i.e.,  on the ``shape'' of the tails of the  L\'evy measure $\mu$.
\end{example}

\begin{example}\label{exa2}  When the L\'evy measure $\mu$ is ``extremely heavy-tailed'' in the sense explained above, the function $\mathfrak{m}$ may fail to belong to the class $\Ld$ \, at all.  Consider the measure
$$
\mu(dx)= \delta_{-1}(dx)+ \sum_{k\geq 0} 2^{k-2k/(1-2H)} \delta_{2^{k}}(dy).
$$
Then  $\mu(\real_-)>0$, \eqref{m2r} fails, whereas \eqref{esint} is satisfied:
$$
\int_\real |y|^{2/(3-2H)} \mu(dy)= 1+ \sum_{k\geq 0} 2^{2k/(3-2H)} 2^{k-2k/(1-2H)}
=
1+ \sum_{k\geq 0} 2^{-k\cdot \frac{1+4H(2-H)}{(3-2H)(1-2H)}}<\infty.
$$
Using \eqref{c1}, we can write $\mathfrak{m} (r)$ explicitly:
$$
\mathfrak{m} (r)= -\ell(-r) + c_H r^{-(3-2H)/(1-2H)} \sum_{k:\, 2^k \leq r\Gamma(H+1/2)} 2^k, \quad r\geq 0,
$$
where $c_H$ is defined in \eqref{CH}. To shorten the notation, put $c:= (\Gamma(H+1/2))^{-1}$.
We have for $r_n:= 2^n c$  and $r_n':= (2^n-1)c$, respectively,
$$
\mathfrak{m} (r_n) = \mathfrak{m}(2^nc) =-\ell (-2^nc) +c_H (2^n c)^{- (3-2H)/(1-2H)} \sum_{k=0}^n 2^k
$$
and
$$
\mathfrak{m} (r_n')  =   \mathfrak{m}((2^n-1)c) =-\ell (-(2^n-1)c) +c_H ((2^n-1) c)^{- (3-2H)/(1-2H)} \sum_{k=0}^{n-1} 2^k.
$$
Note that
$$
\ell (-(2^n-1)c)\sim \ell (-2^nc) \sim -c_H (-2^n c)^{- (3-2H)/(1-2H)} \quad\text{and}\quad  \sum_{k=0}^n 2^k \sim 2^{n+1}
$$
as $n\to\infty$.
Therefore,
$$
\lim_{n\to\infty} \frac{\mathfrak{m}((2^n-1)c)}{\mathfrak{m}(2^nc)} =  \frac{1}{2},
$$
and $\mathfrak{m}\notin\Ld$.
\end{example}

From these two examples one can see that in the ``extremely heavy-tailed'' case the asymptotic behaviour of the distribution density of  $Z_t^H$ is more sensitive with respect to the behaviour of the ``tails'' of the L\'evy measure $\mu$ than in the case where the integrability condition (\ref{m2r}) holds true. If these tails are regularly varying, then (\ref{main-es2}) holds true with the right hand side depending both on $t$ and on the ``shape'' of the ``tails'' of $\mu$. On the other hand, when the ``tails'' of $\mu$ are both  ``heavy'' and ``irregular'', the function $\mathfrak{m}$ may fail to belong to the class $\Ld$, which means that we can not apply Theorem \ref{t-main2} at all.

\section*{Appendix I: Existence of integral  (\ref{yt})}
\addcontentsline{toc}{section}{Appendix}

\begin{proposition}\label{p21} Let  $0<H<1, H\not=1/2$. Then the integral \eqref{yt} is well defined for every $t\in \Re$ if, and only if, the L\'evy measure $\mu$  satisfies (\ref{esint}).

\end{proposition}

\begin{proof}
We consider the case $0<H<1/2$,    the calculations in the case  $1/2<H <1$ are  analogous.   We check the necessary and sufficient condition for the existence of \eqref{yt}  formulated in  \cite[Theorem~2.7]{RR89}. In our case these conditions can be rewritten as
\begin{equation}
\int_{-\infty}^1 f^2(s)ds<\infty, \label{exi3}
\end{equation}
\begin{equation}
I_1:= \int_{-\infty}^1 \int_\Re \left( 1\wedge |f(s)x|^2\right) \mu(dx)ds<\infty,\label{exi1}
\end{equation}
and
    \begin{equation}
    I_2:= \int_{-\infty}^1 \Big| \int_\Re \left( \tau(f(s)x) -f(s)\tau(x)\right) \mu(dx)\Big| ds<\infty,\label{exi2}
    \end{equation}
 where
\begin{equation}
\tau(x)= \begin{cases} x,& \text{if} \quad  |x|\leq 1,
\\ \frac{x}{|x|}, & \text{if} \quad |x|>1,
\end{cases}
\end{equation}
and  $f(s):=f(1,s)$. Clearly, \eqref{exi3}  is satisfied. We  show that a) \eqref{exi1} and \eqref{esint} are equivalent, b)  \eqref{exi2} follows from \eqref{esint}.

a) Split
\begin{equation}
I_1= I_{11}+I_{12}+ I_{13}+ I_{14},\label{s1}
 \end{equation}
 where
 $$
 I_{11}:= \int_{-\infty}^{-1} \int_{|uf(s)|\leq 1}...,  \quad   I_{12}:= \int_{-1}^{1} \int_{|uf(s)|\leq 1}... ,
 $$
 $$
 I_{13}:= \int_{-\infty}^{-1} \int_{|uf(s)| > 1}... , \quad
 I_{14}:= \int_{-1}^{1} \int_{|uf(s)| >1}... ,
  $$
 and estimate the integrals $I_{1i}$, $i=1,..,4$, separately.

Since
$f(s)\sim - \frac{2H-1}{2\Gamma(H+1/2)} |s|^{H-3/2}$ as $s\to-\infty$,  $f(s)\sim -\frac{1}{\Gamma(H+1/2)}|s|^{H-1/2}$, as $s\to 0-$, and $f(s)=\frac{1}{\Gamma(H+1/2)}(1-s)^{H-1/2} $ for $0\leq s<1$,  to check the finiteness of $I_1$  it is enough to substitute $f(s)$ in  the regions $(-\infty, -1]$ and  $(-1,0)$ with, respectively,  $-|s|^{H-3/2}$ and $-|s|^{H-1/2}$, and to check the finiteness of the integrals $\tilde{I}_{11}:= \int_1^\infty \int_{|x|\leq s^{3/2-H}}...$,  $\tilde{I}_{12}:= \int_{0}^1 \int_{|x|\leq s^{1/2-H}}...$, $\tilde{I}_{13}:= \int_1^\infty \int_{|x|> s^{3/2-H}}...$, and $\tilde{I}_{14}:= \int_0^1\int_{|x|> s^{1/2-H}}...$.

We get:
\begin{align*}
\tilde{I}_{11}&= \int_1^\infty   \frac{1}{s^{3-2H}} \int_{|x|\leq s^{3/2-H}} |x|^2 \mu(dx)ds
\\&
= \int_{1}^{\infty}  \frac{1}{s^{3-2H}} \int_{|x|\leq 1} |x|^2 \mu(dx)ds+ \int_{1}^{\infty}  \frac{1}{s^{3-2H}} \int_{1< |x|\leq s^{3/2-H}} |x|^2 \mu(dx)ds
\\&
=\frac{1}{2-2H} \left(\int_{|x|\leq 1} |x|^2 \mu(dx) +  \int_{|x|\geq 1}|x|^{2/(3-2H)} \mu(dx)\right);
\end{align*}
\begin{align*}
\tilde{I}_{12}&= \int_{0}^1 s^{2H-1} \int_{|x|\leq s^{1/2-H}} |x|^2 \mu(dx)ds \leq \frac{1}{2H}\int_{|x|\leq 1 } |x|^2\mu(dx);
\end{align*}
\begin{align*}
\tilde{I}_{13}&=\int_1^\infty  \int_{|x|\geq s^{3/2-H} } \mu(dx)ds
= \int_{|x|\geq 1} (|x|^{2/(3-2H)}-1) \mu(dx);
\end{align*}
\begin{align*}
\tilde{I}_{14}&= \int_{0}^1 \int_{|x|\geq s^{1/2-H}} \mu(dx)ds
=  \int_{|x|\leq 1} |x|^{2/(1-2H)} \mu(dx) +  \int_{|x|\geq 1} \mu(dx).
\end{align*}
Therefore, $I<\infty$ if and only if \eqref{esint} holds true.

b) Split  $I_2:= I_{21}+I_{22}$,
    where     $I_{21}:= \int_{-\infty}^{1} \int_{|uf(s)|\leq 1}... $ and   $I_{22}:= \int_{-\infty}^{1} \int_{|uf(s)|\geq 1}... $.

Observe that
    $$
    \int_{|x|\leq 1/|u| } (ux-u \tau(x)) \mu(dx)  =\int_{1\leq |x|\leq 1/|u|} (ux-\frac{ux}{|x|}) \mu(dx)\leq  2\int_{1\leq |x|\leq 1/|u|}|ux|\mu(dx).
    $$
Then
\begin{equation}
I_{21}\leq 2 \int_{-\infty}^{-1}\int_{1\leq |x|\leq 1/|f(s)| } |f(s)x|\mu(dx)ds.\label{i21}
\end{equation}
To estimate the right-hand side of \eqref{i21} it is enough to estimate
    \begin{align*}
    \tilde{I}_{21}&:= \int_1^\infty \int_{1\leq|x|\leq s^{3/2-H}} \frac{|x|}{s^{3/2-H}} \mu(dx)ds
    =\frac{2}{1-2H} \int_{|x|\geq 1} |x|^{2/(3-2H)}  \mu(dx).
       \end{align*}
Thus, \eqref{esint} implies the finiteness of $\tilde{I}_{21}$, and, consequently, of \eqref{i21}.

To estimate $I_{22}$ observe that
    \begin{align*}
    \int_{|x|\geq 1/|u|} &( \frac{xu}{|xu|}-u\tau(x)) \mu(dx)
    \\&=
    \int_{|x|\geq \max(1/|u|,1)} ( \frac{xu}{|xu|}-u\frac{x}{|x|}) \mu(dx)
    + \int_{1/|u|\leq |x|\leq 1} (\frac{xu}{|xu|}-ux)\mu(dx).
        \end{align*}
Then
    \begin{align*}
    I_{22} &\leq 2\Big(\int_{-\infty}^{-1} \int_{|x|\geq 1/|f(s)|} \mu(dx)ds+ \int_{-1}^{1} \int_{|x|\geq 1} |f(s)|\mu(dx)ds
    \\&
    \quad \quad \quad + \int_{-1}^{1} \int_{1/|f(s)|\leq |x|\leq 1} |f(s)|\mu(dx)ds\Big)
    \\&
    \leq C_1 \Big( \int_1^\infty \int_{|x|\geq s^{3/2-H}} \mu(dx)ds  + \int_{-1}^1 |f(s)|ds\int_{|x|\geq 1} \mu(dx)
    \\&
    \quad \quad \quad +
    \int_0^1 \int_{s^{1/2-H}\leq |x|\leq 1}s^{H-1/2}\mu(dx)ds\Big)
    \\&
    \leq  C_2 \Big( \int_{|x|\geq 1} |x|^{2/(3-2H)}\mu(dx)+  \int_{|x|\geq 1}\mu(dx)    + \int_{|x|\leq 1} |x|^{(1+2H)/(1-2H)}\mu(dx)\Big),
    \end{align*}
and the finiteness of the right-hand side  is implied by  \eqref{esint}.
\end{proof}

\section*{Appendix II: Existence of the distribution density}\label{sB}
\addcontentsline{toc}{section}{Appendix}

\begin{proposition}\label{p22}  Let  $0<H<1, H\not=1/2$. Then, under (\ref{esint}) and our standing assumption (\ref{mu0}), the integral (\ref{yt}) possesses for any $t\not=0$  the distribution density $p_t\in C_b^\infty(\real)$.
\end{proposition}

\begin{remark}  In the non-Markov case  $H\not=1/2$, the kernel $f(t,s)$ provides a strong ``smoothifying'' effect in the sense that the  weakest possible  non-degeneracy assumption (\ref{mu0}) is already sufficient for the integral (\ref{yt}) to possess a smooth distribution density. We refer to \cite{KK11}, Section 3, for the detailed discussion of various forms of the ``smoothifying'' effect for L\'evy driven stochastic integrals with deterministic kernels.
\end{remark}
For the proof we use the following statement, see \cite[Lemma 3.3]{KK11}.

\begin{proposition} \label{p31}
a) For a positive function $h(s)$  having a continuous non-zero derivative on some interval $[a,b]\subset \Re$, one has \begin{equation}
    \int_a^b \big(1-\cos (h(s)x)\big) ds\geq c (x^2\wedge 1). \label{exh1}
    \end{equation}
b) For a positive convex on $(-\infty,b)\subset \Re$ function $h(s)$, satisfying
\begin{equation}
\lim_{s\to -\infty} e^{-\gamma s}h(s)=+\infty\quad \text{ for all}\quad \gamma>0,\label{hlim}
\end{equation}
 one has
    \begin{equation}
    \int_{-\infty}^b \left(1-\cos (x h(s))\right)ds \geq c \ln |x|\label{exh2}
    \end{equation}
for all  $c>0$ and $|x|$ big enough.
\end{proposition}

\begin{proof}[Proof of Proposition~\ref{p22}] Recall (cf. \eqref{chint}) that the characteristic function of $Z_t^H$ is of  the form $\phi(t,z)=e^{\Psi(t,-z)}$.
 For a fixed $t$, the function $h(s)=- t^{H-1/2} f(s)$ satisfies \eqref{hlim}
 with $b=0$. Since
 $\mu(\Re)>0$ (cf. \eqref{mu0}),  there exists $q>0$ such that
$$
Q:= \max \{ \mu((-\infty, -q]), \mu([q,\infty)\}>0.
$$
 Then using (\ref{exh2}) for $|z|$ large enough we get
$$
    -\mathrm{Re}\,\Psi (t,-z) \geq  t \int_{-\infty}^0 \int_{|u|\geq q} (1-\cos ( t^{H-1/2} f(s)u z))\mu(du)ds
     \geq t c Q \ln |q t^{H-1/2}z|.
$$
Since   $c>0$ is arbitrary, the function    $|z|^n|\phi(t,z)|=e^{\mathrm{Re}\Psi(t,-z)+n\ln |z|}$ is integrable in $z$ for any $n\geq 1$.
Therefore  the density $p_t$ is well defined and belongs to $C_b^\infty$ as the inverse Fourier transform of $\phi(t,z)$:
    \begin{equation}
    p_t(x)=\frac{1}{2\pi} \int_\Re  e^{-izx} \phi(t,z)dz. \label{tp0}
    \end{equation}
    \end{proof}

\begin{remark}\label{rb2}
 Literally the same argument  implies that, if the truncation level $\lambda>0$ is chosen small enough,  then $|z|^n|\phi_1(t,z)|=e^{\mathrm{Re}\Psi_1(t,-z)+n\ln |z|}$ is integrable in $z$ for any $a>0$, $t>0$, see (\ref{phi11}). This implies the existence of $\tilde{p}_t(x)$, see (\ref{tp}).
\end{remark}

\textbf{Acknowledgement.}  The authors thank the referee for helpful remarks, and gratefully acknowledge the DFG Grant Schi~419/8-1. The first-named author also gratefully acknowledges the Scholarship of the President of Ukraine for young scientists (2011-2013).

\end{document}